%% file: paper.tex
\documentclass{siamltex}
\usepackage{algorithm,amsmath,xcolor,hyperref}
\usepackage{algorithmic,verbatim,url,comment}
\usepackage{amssymb,amsfonts,amscd,stmaryrd}
\usepackage{paralist,cite}
\usepackage{blkarray}
\usepackage{IlseCmds}
\title{Low-Rank Matrix Approximations Do Not Need\\
a Singular Value Gap\thanks{The work of the first author was 
supported in part by NSF grants IIS-1302231 and NSF IIS-1447283.}}

\author{Petros Drineas\thanks{Department of Computer Science, Purdue University, 
West Lafayette, IN, pdrineas@purdue.edu}
\and
Ilse C. F. Ipsen\thanks{Department of Mathematics, North Carolina State University, Raleigh, NC, ipsen@ncsu.edu}
}

\begin{document}
\maketitle

\begin{abstract}
This is a systematic investigation into the sensitivity of low-rank approximations of real matrices.
We show  that the low-rank approximation errors, in the two-norm, Frobenius norm and more generally,
any Schatten $p$-norm, are insensitive to
additive rank-preserving perturbations in the projector basis; and to matrix perturbations that are additive
or change the number of columns (including multiplicative perturbations). Thus, low-rank matrix approximations
are always well-posed and do not require a singular value gap. 
In the presence of a singular value gap, connections are established between low-rank approximations and
subspace angles.
\end{abstract}


\begin{AM}
15A12, 15A18, 15A42,  65F15, 65F35.
\end{AM}

\input{sec1}

\input{sec2}

\input{sec3}

\appendix
\input{secapp}

\bibliography{Notes}
\end{document}

%% file: sec1.tex
\section{Introduction}
An emerging problem in Theoretical Computer Science and Data Science is 
the low-rank approximation $\bZ\bZ^T\bA$ of a matrix $\bA\in\rmn$ by means of an orthonormal 
basis $\bZ\in\real^{m\times k}$ \cite{Drineas2016,TCS-060}.

The ideal low-rank approximation consists
of the left singular vectors $\bU_k$ associated with the $k$ dominant singular values 
$\sigma_1(\bA)\geq \cdots \geq \sigma_k(\bA)$ of $\bA$, because the low-rank approximation error in the two-norm
is minimal and equal to the first neglected singular value, $\|(\bI-\bU_k\bU_k^T)\bA\|_2=\sigma_{k+1}(\bA)$.
Low-rank approximation $\bZ$ can be determined with subspace iteration or a Krylov space 
method \cite{G2014,MM2015}, with bounds for
 $\|(\bI-\proj)\bA\|_{2,F}$ that contain $\sigma_{k+1}(\bA)$ as an additive
or multiplicative factor. Effort has been put into deriving bounds that not depend on the existence of the 
singular value gap $\sigma_k(\bA)-\sigma_{k+1}(\bA)>0$.    

A closely related problem in numerical linear algebra is the approximation of the dominant subspace
proper \cite{Saa80,Saa11},
that is, computing an orthonormal basis $\bZ\in \real^{m\times k}$ whose space is close to the 
dominant subspace $\range(\bU_k)$. Closeness here means that the sine of the largest principal angle
between the two spaces, $\|\sin{\bTheta}(\bZ,\bU_k)\|_2=\|\bZ\bZ^T-\bU_k\bU_k^T\|_2$ is small. 
For the dominant subspace $\bU_k$ to be well-defined, the associated singular values 
must be separated from the remaining singular values, and there must be a gap
$\sigma_k(\bA)-\sigma_{k+1}(\bA)>0$, see
\cite{Ips99b,Ste73,STS90,Wedin1972,Wedin1983,ZKny13} which are all based on the
perturbation results for invariant subspaces of Hermitian matrices \cite{DaK69,DaK70}.

The purpose of our paper, following up on \cite{DMKI16},
is to establish a clear distinction between the mathematical problems of low-rank approximation, and approximation of dominant subspaces. In particular we show that low-rank approximations are well-defined and well-conditioned, 
by deriving bounds for the low-rank approximation error $(\bI-\bZ\bZ^T)\bA$
in the two-norm, Frobenius norm, and more generally, any Schatten $p$-norm. 
We establish relationships between the mathematical problems of dominant subspace computation
and of low-rank approximation.

\paragraph{Overview}
After setting the notation for the singular value decomposition (Section~\ref{s_sv}),
and reviewing Schatten $p$-norms (Section~\ref{s_schatten}) and angles between subspaces (Section~\ref{s_angles}),
we highlight the main results (Section~\ref{s_high}), followed by proofs for low-rank approximations
(Section~\ref{s_sensla}) and subspace angles (Section~\ref{s_bounds}, Appendix~\ref{s_app}).

\subsection{Singular Value Decomposition (SVD)}\label{s_sv}
Let the non-zero matrix $\bA\in\rmn$ have a full SVD $\bA=\bU\bSigma\bV^T$, where
$\bU\in\rmm$ and $\bV\in\rnn$ are orthogonal matrices, i.e.\footnote{The superscript~$T$
denotes the transpose.}
$$\bU\bU^T=\bU^T\bU=\bI_m, \qquad \bV\bV^T=\bV^T\bV=\bI_n,$$
and $\bSigma\in\rmn$ a diagonal matrix with diagonal elements 
\begin{eqnarray}\label{e_svo}
\|\bA\|_2=\sigma_1(\bA)\geq \dots\geq\sigma_r(\bA)\geq 0,\qquad r\equiv\min\{m,n\}.
\end{eqnarray}
For $1\leq k\leq \rank(\bA)$, the respective leading $k$  columns of $\bU$ and $\bV$ are
 $\bU_k\in\real^{m\times k}$ and $\bV_k\in\real^{m\times k}$. They
are orthonormal, $\bU_k^T\bU_k=\bI_k=\bV_k^T\bV_k$, and 
are associated with the $k$ dominant singular values 
$$\bSigma_k\equiv\diag\begin{pmatrix}\sigma_1(\bA) & \cdots & \sigma_k(\bA)\end{pmatrix}\in\real^{k\times k}.$$ 
Then 
\begin{eqnarray}\label{e_bestrank}
\bA_k\ \equiv \ \bU_k\bSigma_k\bV_k^T \ = \ \bU_k\bU_k^T\bA
\end{eqnarray}
is a best rank-$k$ approximation of $\bA$, and satisfies in the two norm and  in the Frobenius norm, respectively,
$$\|(\bI-\bU_k\bU_k^T)\bA\|_{2,F}\ =\ \|\bA-\bA_k\|_{2,F} \ = \ \min_{\rank(\bB)=k}{\|\bA-\bB\|_{2,F}}.$$
 
\paragraph{Projectors}
We construct orthogonal projectors to capture the target space, which is a dominant subspace of $\bA$.

\begin{definition}\label{d_proj}
A matrix $\proj\in\rmm$ is an {\rm orthogonal projector}, if it is idempotent and symmetric,
\begin{eqnarray}\label{e_dproj}
\proj^2\ =\ \proj \ = \ \proj^T.
\end{eqnarray}
\end{definition}

For $1\leq k\leq \rank(\bA)$,
the matrix $\bU_k\bU_k^T=\bA_k\bA_k^{\dagger}$ is the orthogonal projector onto the
$k$-dimensional dominant subspace $\range(\bU_k)=\range(\bA_k)$.
Here the pseudo inverse is $\bA_k^{\dagger}=\bV_k \bSigma_k^{-1}\bU_k^T$.

\subsection{Schatten $p$-norms}\label{s_schatten} 
These are norms defined on the singular values of real and complex matrices, and thus special cases of symmetric 
gauge functions.
We briefly review their properties, based on \cite[Chapter IV]{Bhatia1997} and \cite[Sections 3.4-3.5]{HoJ91}.

\begin{definition}\label{d_norms}
For integers $p\geq 1$, the {\rm Schatten  $p$ norms} on $\rmn$ are
$$\spn{\bA}\ \equiv \ \sqrt[p]{\sigma_1(\bA)^p+\cdots+\sigma_r(\bA)^p}, \qquad 
 r\equiv\min\{m,n\}.$$
\end{definition}

\paragraph{Popular Schatten norms:}
\begin{description}
\item[$\qquad p=1:$\  ] Nuclear (trace) norm
$\ \|\bA\|_* \ =\ \sum_{j=1}^r{\sigma_j(\bA)}\ = \son{\bA}$.
\item[$\qquad p=2:$\ ] Frobenius norm
$\ \|\bA\|_F \ =\ \sqrt{\sum_{j=1}^r{\sigma_j(\bA)^2}}\ = \  \stn{\bA}$.
\item[$\qquad p=\infty:$\ ] Euclidean (operator) norm
$\ \|\bA\|_2\ =\ \sigma_1(\bA)\ = \ \sinf{\bA}$.
\end{description}
\bigskip

We will make ample use of the following properties.
\begin{lemma}
Let $\bA\in\rmn$, $\bB\in\real^{n\times \ell}$, and $\bC\in\real^{s\times m}$.
\begin{compactitem}
\item Unitary invariance:\\
If $\bQ_1\in\real^{s\times m}$ with $\bQ_1^T\bQ_1=\bI_m$ and
$\bQ_2\in\real^{\ell \times n}$ with $\bQ_2^T\bQ_2=\bI_n$, then
$$\spn{\bQ_1\bA\bQ_2^T}\ = \ \spn{\bA}.$$

\item Submultiplicativity: $\ \spn{\bA\bB}\ \leq \ \spn{\bA} \spn{\bB}$.

\item Strong submultiplicativity (symmetric norm):
$$\spn{\bC\bA\bB}\ \leq \ \sigma_1(\bC)\,\sigma_1(\bB)\,\spn{\bA} \ = \ \|\bC\|_2\,\|\bB\|_2\,\spn{\bA}.$$

\item Best rank-$k$ approximation: 
$$\spn{(\bI-\bU_k\bU_k^T)\bA}\ =\ \spn{\bA-\bA_k}\ = \ \min_{\rank(\bB)=k}{\spn{\bA-\bB}}$$
\end{compactitem}
\end{lemma}

\subsection{Principal Angles between Subspaces}\label{s_angles}
We review the definition of angles between subspaces, and the connections between angles and projectors.

\begin{definition}[Section 6.4.3 in \cite{GovL13} and Section 2 in \cite{ZKny13}]
Let $\bZ\in\real^{m\times k}$ and $\hZ\in\real^{m\times \ell}$ with $\ell\geq k$
have orthonormal columns so that $\bZ^T\bZ=\bI_k$ and $\hZ^T\hZ=\bI_{\ell}$.
 The singular values of $\bZ^T\hZ$ are the diagonal elements of the $k\times k$ diagonal matrix
$$\cos{\bTheta}(\bZ,\hZ)\equiv\diag\begin{pmatrix}\cos{\theta_1}&\cdots &\cos{\theta_k}\end{pmatrix},$$
where $\theta_j$ are the {\rm principal (canonical) angles} between $\range(\bZ)$ and $\range(\hZ)$.
\end{definition}

Next we show how to extract the principal angles between two subspaces of possibly different dimensions,
we make use of projectors.

\begin{lemma}\label{l_projbasis}
Let $\proj\equiv \bZ\bZ^T$ and $\projt\equiv \hZ\hZ^T$ be orthogonal projectors, where
$\bZ\in\real^{m\times k}$ and $\hZ\in\real^{m\times \ell}$ have orthonormal columns. Let $\ell\geq k$, and  define
$$\sin{\bTheta}(\proj,\projt)=\sin{\bTheta}(\bZ,\hZ)\
\equiv\ \diag\begin{pmatrix}\sin{\theta_1}&\cdots &\sin{\theta_k}\end{pmatrix},$$
where $\theta_j$ are the $k$ principal angles between $\range(\bZ)$ and $\range(\hZ)$.
\begin{enumerate}
\item If  $\rank(\hZ)=k=\rank(\bZ)$, then
\begin{eqnarray*}
\spn{\sin{\bTheta}(\bZ,\hZ)}&=&\spn{(\bI-\proj)\,\projt}=\spn{(\bI-\projt)\,\proj}.
\end{eqnarray*}
In particular
$$\|(\bI-\proj)\,\projt\|_2=\|\proj-\projt\|_2\leq 1$$
represents the distance between the subspaces $\range(\proj)$ and $\range(\projt)$.
\item If $\rank(\hZ)>k=\rank(\bZ)$, then
$$\spn{\sin{\bTheta}(\bZ,\hZ)}=\spn{(\bI-\projt)\,\proj}\ \leq \ \spn{(\bI-\proj)\,\projt}.$$
\end{enumerate}
\end{lemma}

\begin{proof}
The two-norm expressions follow from \cite[Section~2.5.3]{GovL13} and \cite[Section~2]{Wedin1983}.
The Schatten $p$-norm expressions follow from the CS decompositions in 
\cite[Theorem 8.1]{PaigeWei94}, \cite[Section 2]{ZKny13}, and Section~\ref{s_app}.
 \end{proof}

\subsection{Highlights of the Main Results}\label{s_high}
We present a brief overview of the main results: The well-conditioning of low-rank
approximations under additive perturbations in projector basis and the matrix (Section~\ref{s_low});
the well-conditioning of low-rank approximations under perturbations that change the matrix dimension
(Section~\ref{s_dimchange});
and the connection between low-rank approximation errors and angles between subspaces (Section~\ref{s_angle}).

Thus: Low-rank approximations are well-conditioned, and don't need a gap.

\subsubsection{Additive perturbations in the projector basis and the matrix}\label{s_low}
We show that the low-rank approximation error is insensitive to
additive rank-preserving perturbations in the projector basis (Theorem~\ref{t_a} and Corollary~\ref{c_a}),
and to additive perturbations in the matrix (Theorem~\ref{t_ba} and Corollary~\ref{c_ba}). 

We start with perturbations in the projector basis.

\begin{mytheorem}[Additive rank-preserving perturbations in the projector basis]\label{t_a}
Let $\bA\in\rmn$; let $\bZ\in\real^{m\times\ell}$ be a projector basis with orthonormal columns
so that $\bZ^T\bZ=\bI_{\ell}$;
and let $\hZ\in\real^{m\times \ell}$ be  its perturbation with
$$\epsilon_Z \equiv\ \|\hZ^{\dagger}\|_2\,\|\bZ-\hZ\|_2\  = \
\underbrace{\|\hZ\|_2\|\hZ^{\dagger}\|_2}_{\mathrm{Deviation~from} \atop \mathrm{orthonormality}}
\underbrace{\frac{\|\hZ-\bZ\|_2}{\|\hZ\|_2}.}_{\mathrm{Relative~distance} \atop \mathrm{from~exact~basis}}$$
\begin{compactenum}
\item If $\rank(\hZ)=\rank(\bZ)$ then
\begin{eqnarray*}
\spn{(\bI-\bZ\bZ^T)\bA} - \epsilon_Z\,\spn{\bA} & \leq &\spn{(\bI-\hZ\hZ^{\dagger})\bA}\\
&&\qquad\qquad \leq \spn{(\bI-\bZ\bZ^T)\bA}+ \epsilon_Z\,\spn{\bA}.
\end{eqnarray*}
\item If $\|\bZ-\hZ\|_2\leq 1/2$, then $\rank(\hZ)=\rank(\bZ)$ and $\epsilon_Z\leq 2\,\|\bZ-\hZ\|_2$.
\end{compactenum}
\end{mytheorem}

\begin{proof}
See Section~\ref{s_sensla}, and in particular Theorem~\ref{t_perturbla}.
\end{proof}

Theorem~\ref{t_a} bounds the change in the absolute approximation error
in terms of the additive perturbation $\epsilon_Z$ amplified by the norm of $\bA$.
The term $\epsilon_Z$ can also be written as the  product of  two factors: 
(i) the two-norm condition number $\|\hZ\|_2\|\hZ^{\dagger}\|_2$ of the perturbed basis with regard to (left) inversion;
(ii) and relative two-norm distance between the bases.
The assumption here is that the perturbed vectors $\hZ$
are linearly independent, but not necessarily orthonormal. Hence the Moore
Penrose inverse replaces the transpose in the orthogonal projector,
and the condition number represents the deviation of $\hZ$ from orthonormality.
The special case $\|\bZ-\hZ\|_2\leq 1/2$ implies both that the perturbed projector
basis is well-conditioned and that it is close to the exact basis.

The lower bound in Theorem~\ref{t_a} simplifies when the columns of $\bZ$ are dominant singular vectors of~$\bA$. No singular value gap is required below, as we merely pick 
the leading $k$ columns of $\bU$ from some SVD of $\bA$, and then perturb them.
 
\begin{mycorollary}[Rank-preserving perturbation of dominant singular vectors]\label{c_a}
Let $\bU_k\in\real^{m\times k}$ in (\ref{e_bestrank}) be $k$ dominant left singular 
vectors of $\bA$. Let $\hU\in\real^{m\times k}$ be a perturbation of $\bU_k$
with $\rank(\hU)=k$ or $\|\bU_k-\hU\|_2\leq 1/2$; 
and let $\epsilon_U \equiv\ \|\hU^{\dagger}\|_2\,\|\bU_k-\hU\|_2$.
Then
\begin{eqnarray*}
\spn{(\bI-\bU_k\bU_k^T)\bA} &  \leq & \spn{(\bI-\hU\hU^{\dagger})\bA}\leq
\spn{(\bI-\bU_k\bU_k^T)\bA} + \epsilon_U\,\spn{\bA}.
\end{eqnarray*}
\end{mycorollary}

Next we consider perturbations in the matrix, with a bound  that is  completely general and
holds for any projector $\proj$ in any Schatten $p$-norm.

\begin{mytheorem}[Additive perturbations in the matrix]\label{t_ba}
Let $\bA$, $\bA+\bE\in\rmn$; and let $\proj\in\rmm$ be an orthogonal projector as in
(\ref{e_dproj}). Then
\begin{eqnarray*}
\bigg|\, \spn{(\bI-\proj)\, (\bA+\bE)}- \spn{(\bI-\proj)\bA}\, \bigg| \ \leq\  \spn{\bE}.
\end{eqnarray*}
\end{mytheorem}

\begin{proof} See Section~\ref{s_sensla}, and in particular Theorem~\ref{t_perturbm}.
\end{proof}

Theorem~\ref{t_ba} shows that the low-rank approximation error is well-conditioned, in the absolute sense,
to additive perturbations in the matrix. 

Theorem~\ref{t_ba} also implies the following upper bound for
a low-rank approximation of $\bA$  by means of singular vectors of $\bA+\bE$.
Again, no singular value gap is required. We merely pick 
the leading $k$ columns $\bU_k$ obtained from some SVD of $\bA$, and the
leading $k$ columns $\hU_k$ obtained from some SVD of $\bA+\bE$.
 
\begin{mycorollary}[Low-rank approximation from additive perturbation]\label{c_ba}
Let $\bU_k\in\real^{m\times k}$  in (\ref{e_bestrank}) be $k$ dominant left singular vectors 
of $\bA$;
and let $\hU_k\in\real^{m\times k}$ be $k$ dominant left singular vectors of $\bA+\bE$.
Then
$$\|(\bI-\bU_k\bU_k^T)\bA\|_2 \ \leq\ 
\|(\bI-\hU_k \hU_k^T)\bA\|_{2}\, \ \leq\ \|(\bI-\bU_k\bU_k^T)\bA\|_2 + 2\|\bE\|_{2}.$$
\end{mycorollary}

\begin{proof}
See Section~\ref{s_sensla}, and in particular Corollary~\ref{c_cba}.
\end{proof}

Bounds with an additive dependence on $\bE$, like the two-norm bound above,
can be derived for other Schatten $p$-norms as well, and can then be combined 
with bounds for $\bE$ in
 \cite{Kundu2014,Achlioptas2007,Drineas2011}
where $\bA+\bE$  is obtained from element-wise sampling from $\bA$. 
 
 \subsubsection{Perturbation that change the matrix dimension}\label{s_dimchange}
 We consider perturbations that can change the number of columns in~$\bA$ and 
include, among others,
multiplicative perturbations of the form $\wA=\bA\bX$. However, our bounds
are completely general
and hold also in the absence of any relation between $\range(\bA)$ and $\range(\wA)$.

Presented below are bounds for the two-norm (Theorem~\ref{t_bb2}), the Frobenius norm (Theorem~\ref{t_bbF})
and general Schatten $p$-norms (Theorem~\ref{t_bbp}). 

\begin{mytheorem}[Two-norm]\label{t_bb2}
Let $\bA\in\rmn$; $\wA\in\real^{m\times c}$; and $\proj\in\rmm$ an orthogonal projector
as in (\ref{e_dproj}) with $\rank(\proj)=c$.
Then
\begin{eqnarray}\label{e_bb21}
\bigg|\|(\bI-\proj)\,\bA\|_{2}^2-  \|(\bI-\proj)\,\wA\|_{2}^2\bigg|  \ \leq \  \|\bA\bA^T-\wA\wA^T\|_2.
\end{eqnarray}
If also $\rank(\wA)=c$ then
\begin{eqnarray}\label{e_bb22}
\|(\bI-\wA\wA^{\dagger})\,\bA\|_{2}^2 \ \leq \  \|\bA\bA^T-\wA\wA^T\|_2.
\end{eqnarray}
If $\wA_k\in\real^{m\times c}$ is a best rank -$k$ approximation of $\wA$ with $\rank(\wA_k)=k<c$ then
\begin{eqnarray}\label{e_bb23}
\|(\bI-\wA_k\wA_k^{\dagger})\,\bA\|_{2}^2 \ \leq\  \|\bA-\bA_k\|_2^2+ 2\,\|\bA\bA^T -\wA\wA^T\|_{2}.
\end{eqnarray}
\end{mytheorem}

\begin{proof}
See Section~\ref{s_sensla}, specifically 
Theorem~\ref{t_perturbmm} for (\ref{e_bb21});
Theorem~\ref{t_lc} for (\ref{e_bb22}); and 
and Theorem~\ref{t_lck} for (\ref{e_bb23}).
\end{proof}

The bounds~(\ref{e_bb23})  are identical to~\cite[Theorem 3]{DKM06}, while
(\ref{e_bb21}) and~(\ref{e_bb22}), though similar in spirit, are novel. 
The bound (\ref{e_bb21}) holds for any orthogonal projector $\proj$, 
in contrast to prior work which was limited to multiplicative perturbations $\wA=\bA\bX$
with bounds for $\|\bA\bA^T-\wA\wA^T\|_2$ for matrices $\bX$ that sample and rescale columns 
\cite{DKM06a,Holodnak2015},   and other constructions of $\bX$  \cite{TCS-060}.

\begin{mytheorem}[Frobenius norm]\label{t_bbF}
Let $\bA\in\rmn$; $\wA\in\real^{m\times c}$; and $\proj\in\rmm$ an orthogonal projector 
as in (\ref{e_dproj})
with $\rank(\proj)=c$.
Then
\begin{eqnarray}
\lefteqn{\bigg|\|(\bI-\proj)\,\wA\|_F^2  -\|(\bI-\proj)\,\bA\|_F^2\bigg| \  \leq}\label{e_bbF1}\\
& &\ \min\bigg\{ \|\bA\bA^T-\wA\wA^T\|_*, \quad \sqrt{m-c}\,\|\bA\bA^T-\wA\wA^T\|_F\bigg\}.\nonumber
\end{eqnarray}
If also $\rank(\wA)=c$ then
\begin{eqnarray}\label{e_bbF2}
\|(\bI-\wA\wA^{\dagger})\,\bA\|_{F}^2 \ \leq
\ \min\bigg\{ \|\bA\bA^T-\wA\wA^T\|_*, \ \sqrt{m-c}\,\|\bA\bA^T-\wA\wA^T\|_F\bigg\}.
\end{eqnarray}
If $\wA_k\in\real^{m\times c}$ is a best rank-$k$  approximation of $\wA$ with $\rank(\wA_k)=k<c$ then
\begin{eqnarray}
(\bI-\wA_k\wA_k^{\dagger})\,\bA\|_{F}^2 \ &\leq&\ \|\bA-\bA_k\|_F^2\label{e_bbF3}\\
&&+2\, \min\big\{ \|\bA\bA^T-\wA\wA^T\|_*, \, \sqrt{m-c}\,\|\bA\bA^T-\wA\wA^T\|_F\big\}.\nonumber
\end{eqnarray}
\end{mytheorem}

\begin{proof}
See Section~\ref{s_sensla}, specifically 
Theorem~\ref{t_perturbmm} for (\ref{e_bbF1});
Theorem~\ref{t_lc} for (\ref{e_bbF2});
and Theorem~\ref{t_lck} for (\ref{e_bbF3}).
\end{proof}

The bound (\ref{e_bbF1}) holds for any $\proj$, and is the first one of its kind in this generality.
The bound (\ref{e_bbF3}) is similar to \cite[Theorem 2]{DKM06}, and weaker for smaller $k$ but tighter for  larger $k$. 

More generally, Theorem~\ref{t_bbF} relates the low-rank approximation error in the Frobenius norm 
with the error  $\bA\bA^T-\wA\wA^T$ in the trace norm, i.e. the Schatten one-norm.
This is a novel connection, and it  should motivate further work into understanding the behaviour
of the trace norm, thereby complementing prior investigations into the two-norm and Frobenius norm.

\begin{mytheorem}[General Schatten $p$-norms]\label{t_bbp}
Let $\bA\in\rmn$; $\wA\in\real^{m\times c}$; and $\proj\in\rmm$ an orthogonal projector 
as in (\ref{e_dproj})
with $\rank(\proj)=c$.
Then
\begin{eqnarray}
\lefteqn{\bigg|\spn{(\bI-\proj)\,\bA}^2-  \spn{(\bI-\proj)\,\wA}^2\bigg| \ \leq\label{e_bbp1}}\\
& &\ \min\bigg\{ \spnt{\bA\bA^T-\wA\wA^T}, \quad \sqrt[p]{m-c}\,\spn{\bA\bA^T-\wA\wA^T}\bigg\}.
\nonumber
\end{eqnarray}
If $\rank(\wA)=c$ then
\begin{eqnarray}
\lefteqn{\spnt{(\bI-\wA\wA^{\dagger})\,\bA}^2 \ \leq}\label{e_bbp2}\\
& & \ \min\bigg\{ \spnt{\bA\bA^T-\wA\wA^T}, \quad \sqrt[p]{m-c}\,\spn{\bA\bA^T-\wA\wA^T}\bigg\}.\nonumber
\end{eqnarray}
If $\wA_k\in\real^{m\times c}$ is a best rank approximation of $\wA$ with $\rank(\wA_k)=k<c$ then
\begin{eqnarray}
\spnt{(\bI-\wA_k\wA_k^{\dagger})\,\bA}^2 \ &\leq&\ \spn{\bA-\bA_k}^2+\label{e_bbp3}\\
&&2\, \min\big\{ \spnt{\bA\bA^T-\wA\wA^T}, \, \sqrt[p]{m-c}\,\spn{\bA\bA^T-\wA\wA^T}\big\}.\nonumber
\end{eqnarray}
\end{mytheorem}

\begin{proof}
See Section~\ref{s_sensla}, specifically 
Theorem~\ref{t_perturbmm} for (\ref{e_bbp1});
Theorem~\ref{t_lc} for (\ref{e_bbp2});
and Theorem~\ref{t_lck} for  (\ref{e_bbp3}).
\end{proof}

Theorem~\ref{t_bbp} is new. To our knowledge, non-trivial bounds for 
$\spn{\bA\bA^T-\wA\wA^T}$ for general~$p$ do no exist.

\subsubsection{Relations between low-rank approximation error and subspace angle}\label{s_angle}
For matrices $\bA$ with a singular value gap,
we bound the low-rank approximation error in terms of the subspace 
angle (Theorem~\ref{t_c}) 
and discuss the tightness of the bounds (Remark~\ref{r_tc}).
The singular value gap is required for the dominant subspace to be well-defined, 
but no assumptions on the accuracy of the low-rank approximation are required.
 
Assume that $\bA \in \rmn$ has a gap after the $k$th singular value,
\begin{equation*}
\|\bA\|_2=\sigma_1(\bA)\geq \cdots\geq \sigma_k(\bA)> 
\sigma_{k+1}(\bA)\geq \cdots \geq \sigma_r(\bA)\geq 0,\qquad
r\equiv\min\{m,n\},
\end{equation*}

Below, we highlight the bounds from Section~\ref{s_bounds} for low-dimensional approximations,
compared to the dimension $m$ of the host space for $\range(\bA$).

\begin{mytheorem}\label{t_c}
Let $\proju\equiv \bA_k\bA_k^{\dagger}$ be the orthogonal projector onto the
dominant $k$-dimensional subspace of $\bA$; and
let $\proj\in\rmm$ be some orthogonal projector as in (\ref{e_dproj}) with
$k\leq \rank(\proj)<m-k$. Then
\begin{equation*}\label{e_tc1}
\sigma_k(\bA)\,\spn{\sin{\bTheta(\proj,\proju)}} \leq \spn{(\bI-\proj)\bA} \leq
\|\bA\|_2\,\spn{\sin{\bTheta(\proj,\proju)}}+\spn{\bA-\bA_k}.
\end{equation*}
\end{mytheorem}

\begin{proof} See Section~\ref{s_bounds}, and in particular Theorems~\ref{t_lau} and~\ref{t_lal1}.
\end{proof}

Theorem~\ref{t_c} shows  that for dominant subspaces of sufficiently low dimension, the approximation error is
bounded by the product of the subspace angle with a dominant singular value.
The upper bound also contains the subdominant singular values.

\begin{myremark}[Tightness of Theorem~\ref{t_c}]\label{r_tc}
\begin{compactitem}
\item If $\rank(\bA)=k$, so that $\bA-\bA_k=\bzero$, then the tightness
depends on the spread of the non-zero singular values,
$$\sigma_k(\bA)\,\spn{\sin{\bTheta(\proj,\proju)}} \leq \spn{(\bI-\proj)\,\bA} \leq
\|\bA\|_2\,\spn{\sin{\bTheta(\proj,\proju)}}.$$

\item If $\rank(\bA)=k$ and $\sigma_1(\bA)= \dots =\sigma_k(\bA)$, then the bounds are tight,
and they are equal to 
$$ \spn{(\bI-\proj)\,\bA} = \|\bA\|_2\,\spn{\sin{\bTheta(\proj,\proju)}}.$$

\item If $\range(\proj)=\range(\proju)$, so that $\sin{\bTheta(\proj,\proju)}=\bzero$,
then the upper bound is tight and equal to
$$\spn{(\bI-\proj)\,\bA} = \spn{\bA-\bA_k}.$$
 \end{compactitem}
 \end{myremark}

%% file: sec2.tex
\section{Well-conditioning of low-rank approximations}\label{s_sensla}
We investigate the effect of additive, rank-preserving perturbations in the projector basis
on the orthogonal projector (Section~\ref{s_opperturb})
and on the low-rank approximation error (Section~\ref{s_lrperturb}); and the effect 
on the low-rank approximation error of matrix perturbations, both additive and dimension changing
(Section~\ref{s_lrmperturb}). We also relate  low-rank approximation
error and error matrix (Section~\ref{s_errmatrix}).

\subsection{Orthogonal projectors, and perturbations in the projector basis}\label{s_opperturb}
We show that orthogonal projectors and subspace angles are insensitive
to additive,  rank-preserving perturbations in the projector basis (Theorem~\ref{t_perturbp})
if the perturbed projector basis is well-conditioned.

\begin{theorem}\label{t_perturbp}
Let $\bA\in\rmn$; let $\bZ\in\real^{m\times s}$ be a projector basis with orthonormal
columns so that $\bZ^T\bZ=\bI_s$; and let $\hZ\in\real^{m\times s}$ be its perturbation.
\begin{compactenum}
\item If $\rank(\hZ)=\rank(\bZ)$, then the {\rm distance} between $\range(\bZ)$ and $\range(\hZ)$ is
\begin{equation*}
\|\bZ\bZ^T-\hZ\hZ^{\dagger}\|_2=\|\sin{\bTheta(\bZ,\hZ)}\|_2\leq \epsilon_Z\equiv\ \|\hZ^{\dagger}\|_2\,\|\bZ-\hZ\|_2.
\end{equation*}
\item If $\|\bZ-\hZ\|_2\leq 1/2$, then $\rank(\hZ)=\rank(\bZ)$ and $\epsilon_Z\leq 2\,\|\bZ-\hZ\|_2$.
\end{compactenum}
\end{theorem}

\begin{proof}
\begin{compactenum}
\item The equality follows from Lemma~\ref{l_projbasis}.
The upper bound follows from \cite[Theorem 3.1]{Ste2011} and \cite[Lemma 20.12]{Higham2002},
but we provide a simpler proof  for this context. Set $\hZ=\bZ+\bF$, and abbreviate
$\proj\equiv\bZ\bZ^T$ and $\projt\equiv\hZ\hZ^{\dagger}$.
Writing
$$(\bI-\proj)\projt=(\bI-\bZ\bZ^T)\hZ\hZ^{\dagger}=
(\bI-\bZ\bZ^T)(\bZ+\hZ-\bZ)\,\hZ^{\dagger}=(\bI-\proj)\bF\,\hZ^{\dagger}$$
gives
\begin{equation}\label{e_sin2}
\|\sin{\bTheta(\bZ,\hZ)}\|_2=\|(\bI-\proj)\projt\|_2\leq \|\hZ^{\dagger}\|_2\,\|\bF\|_2.
\end{equation}
\item To show $\rank(\hZ)=\rank(\bZ)$ in the special case $\|\bZ-\hZ\|_2\leq 1/2$, consider
the singular values $\sigma_j(\bZ)=1$  and $\sigma_j(\hZ)$,  $1\leq j\leq s$.
The well-conditioning of singular values \cite[Corollary~8.6.2]{GovL13} implies
$$\left| 1-\sigma_j(\hZ)\right| = \left| \sigma_j(\bZ)-\sigma_j(\hZ)\right| \leq \|\bF\|_2\leq 1/2,
\qquad 1\leq j\leq s.$$
Thus $\min_{1\leq j\leq s}{\sigma_j(\hZ)}\geq 1/2>0$ and $\rank(\projt)=\rank(\hZ)=s=\rank(\proj)$.
Hence (\ref{e_sin2}) holds with
\begin{equation}\label{e_sin3}
\|\sin{\bTheta(\bZ,\hZ)}\|_2\leq \|\hZ^{\dagger}\|_2\|\bF\|_2\leq 2\|\bF\|_2.
\end{equation}
\end{compactenum}
\end{proof}

Note that $\epsilon_Z$ represents both, an absolute and a relative perturbation as
$$\epsilon_Z =\ \|\hZ^{\dagger}\|_2\,\|\bZ-\hZ\|_2\  = \ \|\hZ\|_2\|\hZ^{\dagger}\|_2\>
\frac{\|\hZ-\bZ\|_2}{\|\hZ\|_2}.$$

\subsection{Approximation errors, and perturbations in the projector basis}\label{s_lrperturb}
We show that low-rank approximation errors are insensitive to additive, rank-preserving perturbations
in the projector basis (Theorem~\ref{t_perturbla}), if the perturbed projector basis is well-conditioned.

\begin{theorem}\label{t_perturbla}
Let $\bA\in\rmn$; let $\bZ\in\real^{m\times s}$ be a projector basis with orthonormal
columns so that $\bZ^T\bZ=\bI_s$; and let $\hZ\in\real^{m\times s}$ be its perturbation with
$\epsilon_Z \equiv\ \|\hZ^{\dagger}\|_2\,\|\bZ-\hZ\|_2$.
\begin{compactenum}
\item If $\rank(\hZ)=\rank(\bZ)$ then
\begin{eqnarray*}
\lefteqn{\spn{(\bI-\bZ\bZ^*)\bA} - \epsilon_Z\,\spn{\bA} \leq  \spn{(\bI-\hZ\hZ^{\dagger})\bA}}
\qquad\qquad\qquad\qquad\qquad\qquad\qquad\qquad\\
& \leq & \spn{(\bI-\bZ\bZ^*)\bA} + \epsilon_Z\,\spn{\bA}.
\end{eqnarray*}
\item If $\|\bZ-\hZ\|_2\leq 1/2$, then  $\rank(\hZ)=\rank(\bZ)$ and $\epsilon_Z\leq 2\, \|\bZ-\hZ\|_2$.
\end{compactenum}
\end{theorem}

\begin{proof}
This is a straightforward consequence of Theorem~\ref{t_perturbp}.

Abbreviate $\proj\equiv \bZ\bZ^T$ and $\projt\equiv\hZ\hZ^{\dagger}$, and write
\begin{equation*}
(\bI-\projt)\bA = (\bI-\proj)\bA +(\proj-\projt)\bA.
\end{equation*}
Apply the triangle and reverse triangle inequalities, strong submultiplicativity,  and
then bound the second summand with Theorem~\ref{t_perturbp},
\begin{equation*}
\spn{(\proj-\projt)\bA}\leq\|\hZ^{\dagger}\|_2\>\|\hZ-\bZ\|_2\spn{\bA} = \epsilon_Z\,\spn{\bA}.
\end{equation*}
\end{proof}

\subsection{Approximation errors, and perturbations in the matrix}\label{s_lrmperturb}
We show that low-rank approximation errors  are insensitive to matrix perturbations that are additive 
(Theorem~\ref{t_perturbm} and Corollary~\ref{c_cba}), and that are dimension changing
(Theorem~\ref{t_perturbmm})

\begin{theorem}[Additive perturbations]\label{t_perturbm}
Let $\bA,\bE\in\rmn$; and let $\proj\in\rmm$ be an orthogonal projector with $\proj^2=\proj=\proj^T$.
Then
\begin{equation*}
\spn{(\bI-\proj)\bA}- \spn{\bE} \leq
\spn{(\bI-\proj)(\bA+\bE)}  \leq \spn{(\bI-\proj)\bA} + \spn{\bE}.
\end{equation*}
\end{theorem}

\begin{proof}
Apply the triangle and reverse triangle inequalities, and the fact that an orthogonal projector has at most unit norm,
$\|\bI-\proj\|_2\leq 1$.
\end{proof}

\begin{corollary}[Low-rank approximation from singular vectors of $\bA+\bE$]\label{c_cba}
Let $\bU_k\in\real^{m\times k}$ in (\ref{e_bestrank}) be $k$ dominant left singular vectors of $\bA$;
and let $\hU_k\in\real^{m\times k}$ be $k$ dominant left singular vectors of $\bA+\bE$.
Then
$$\|(\bI-\bU_k\bU_k^T)\bA\|_2 \ \leq\ 
\|(\bI-\hU_k \hU_k^T)\bA\|_{2}\, \ \leq\ \|(\bI-\bU_k\bU_k^T)\bA\|_2 + 2\|\bE\|_{2}.$$
\end{corollary}

\begin{proof}
Setting  $\proj=  \hU_k \hU_k^T$ in the upper bound of Theorem~\ref{t_perturbm} gives 
\[\|(\bI-\hU_k \hU_k^T)\bA\|_{2}\, \ \leq\ \|(\bI-\hU_k\hU_k^T)\, (\bA+\bE)\|_{2} + \|\bE\|_{2}.\]
Express the approximation errors in terms of singular values,
$$\|(\bI-\hU_k \hU_k^T)\, (\bA+\bE)\|_2=\sigma_{k+1}(\bA+\bE), \qquad
\|(\bI-\bU_k\bU_k^T)\bA\|_2 =\sigma_{k+1}(\bA),$$
apply Weyl's theorem 
\[|\sigma_{k+1}(\bA+\bE)-\sigma_{k+1}(\bA)| \leq \|\bE\|_2,\]
and combine the bounds.
\end{proof}

\begin{theorem}[Perturbations that change the number of columns]\label{t_perturbmm}
Let $\bA\in\rmn$; $\wA\in\real^{m\times c}$; let $\proj\in\rmm$ be an orthogonal projector
as in (\ref{e_dproj}) and $\rank(\proj)=s$; and let $p\geq 1$ an even integer. Then
\begin{compactenum}
\item Two norm $(p=\infty)$
$$\bigg|\|(\bI-\proj)\,\bA\|_{2}^2-  \|(\bI-\proj)\wA\|_{2}^2\bigg|  \ \leq \  \|\wA\wA^T-\bA\bA^T\|_2.$$
\item Schatten $p$ norm $(p$ even$)$
\begin{eqnarray*}
\lefteqn{\bigg|\spn{(\bI-\proj)\,\bA}^2-  \spn{(\bI-\proj)\wA}^2\bigg| \ \leq}\\
& &\ \min\bigg\{ \spnt{\wA\wA^T-\bA\bA^T}, \quad \sqrt[p]{m-s}\,\spn{\wA\wA^T-\bA\bA^T}\bigg\}.
\end{eqnarray*}
\item Frobenius norm  $(p=2)$
\begin{eqnarray*}
\lefteqn{\bigg|\|(\bI-\proj)\,\wA\|_F^2  -\|(\bI-\proj)\,\bA\|_F^2\bigg| \  \leq}\\
& &\ \min\bigg\{ \|\wA\wA^T-\bA\bA^T\|_*, \quad \sqrt{m-s}\,\|\wA\wA^T-\bA\bA^T\|_F\bigg\}.
\end{eqnarray*}
\end{compactenum}
\end{theorem}

\begin{proof}
The proof is motivated by that of  \cite[Theorems 2 and 3]{DKM06}.
The bounds are obvious for $s=m$ where $\proj=\bI_m$, so assume $s<m$.

\paragraph{1. Two-norm}
The invariance of the two norm under transposition and the triangle inequality imply
\begin{eqnarray*}
\|(\bI-\proj)\wA\|_{2}^2&=&\|\wA^T(\bI-\proj)\|_{2}^2=\|(\bI-\proj)\wA\wA^T(\bI-\proj)\|_{2}\\
&=&\|(\bI-\proj)\bA\bA^T(\bI-\proj)\ +\ (\bI-\proj)\left(\wA\wA^T-\bA\bA^T\right)(\bI-\proj)\|_{2}\\
&\leq &\|(\bI-\proj)\bA\bA^T(\bI-\proj)\|_{2} +\|(\bI-\proj)\left(\wA\wA^T-\bA\bA^T\right)(\bI-\proj)\|_{2}.
\end{eqnarray*}
The first summand equals
$$\|(\bI-\proj)\bA\bA^T(\bI-\proj)\|_{2} =\|(\bI-\proj)\bA\|_{2}^2,$$
while the second one is bounded by submultiplicativity and $\|\bI-\proj\|_2\leq 1$,
\begin{eqnarray*}
\|(\bI-\proj)\left(\wA\wA^T-\bA\bA^T\right)(\bI-\proj)\|_{2}
&\leq& \|\bI-\proj\|_2^2\,\|\wA\wA^T-\bA\bA^T\|_{2}\\
&\leq& \|\wA\wA^T-\bA\bA^T\|_{2}.
\end{eqnarray*}
This gives the upper bound
$$\|(\bI-\proj)\wA\|_{2}^2-\|(\bI-\proj)\bA\|_{2}^2\leq \|\wA\wA^T-\bA\bA^T\|_{2}.$$
Apply the inverse triangle inequality to show the lower bound,
$$-\|\wA\wA^T-\bA\bA^T\|_{2}\leq \|(\bI-\proj)\wA\|_{2}^2-\|(\bI-\proj)\bA\|_{2}^2.$$

\paragraph{2. Schatten $p$-norm $(p$ even$)$}
The proof is similar to that of the two-norm, since  an even Schatten $p$-norm
is a \textit{Q-norm} \cite[Definition IV.2.9]{Bhatia1997}, that is, it represents  a quadratic gauge function.
This can be seen in terms of singular values,
$$\spn{\bM}^{p}=\sum_{j}{\left(\sigma_j(\bM)\right)^{p}}
=\sum_{j}{\left(\sigma_j(\bM\bM^T)\right)^{p/2}}=\spnt{\bM\bM^T}^{p/2}.$$
Hence
\begin{equation}\label{e_qnorm}
\spn{\bM}^2\ = \ \spnt{\bM\bM^T}.
\end{equation}
Abbreviate $\bM\equiv \wA\wA^T-\bA\bA^T$, and $\bB\equiv \bI-\proj$ where $\bB^T=\bB$ and $\|\bB\|_2=1$.
Since singular values do not change under transposition, 
it follows from (\ref{e_qnorm}) and the triangle inequality that
\begin{eqnarray*}
\spn{\bB\,\wA}^2&=&\spn{\wA^T\,\bB}^2=\spnt{\bB\,\wA\wA^T\,\bB}
=\spnt{\bB\bA\bA^T\bB\ +\ \bB\,\bM\,\bB}\\
&\leq &\spnt{\bB\,\bA\bA^T\,\bB} +\spnt{\bB\,\bM\,\bB}.
\end{eqnarray*}
Apply (\ref{e_qnorm}) to the first summand, 
$\ \spnt{\bB\,\bA\bA^T\,\bB} =\spn{\bB\,\bA}^2$,
and insert it into the above inequalities,
\begin{equation}\label{e_24}
\spn{\bB\,\wA}^2-\spn{\bB\,\bA}^2\leq\spnt{\bB\,\bM\,\bB}.
\end{equation}
\begin{compactenum}
\item First term in the minimum:
Bound (\ref{e_24}) with strong submultiplicativity and $\|\bB\|_2= 1$,
\begin{eqnarray*}
\spn{\bB\,\bM\,\bB}\leq \|\bB\|_2^2\,\spn{\bM}\leq \spn{\bM},
\end{eqnarray*}
which gives the upper bound
$$\spn{\bB\,\wA}^2-\spn{\bB\,\bA}^2\leq \spn{\bM}.$$
Apply the inverse triangle inequality to show the lower bound
$$-\spn{\bM}\leq \spn{\bB\,\wA}^2-\spn{\bB\,\bA}^2.$$
\item Second term in the minimum: From
$$\rank\left(\bB\,\bM\,\bB\right)\leq \rank(\bB)=\rank(\bI-\proj)=m-s>0$$
follows $\sigma_j(\bB)=1$, $\leq j\leq m-s$.
With the non-ascending singular value ordering in (\ref{e_svo}),
the Schatten $p$-norm needs to sum over only the largest $m-s$ singular values.
This, together with the singular value inequality \cite[(7.3.14)]{HoJ12}
$$\sigma_j(\bB\,\bM\,\bB)\leq \sigma_1(\bB)^2\,\sigma_j(\bM) = 1\cdot \,\sigma_1(\bM), \qquad 1\leq j\leq m-s, $$
gives
$$\spnt{\bB\,\bM\,\bB}^{p/2}=\sum_{j=1}^{m-s}{\left(\sigma_j(\bB\,\bM\,\bB)\right)^{p/2}}
\leq \sum_{j=1}^{m-s}{1\cdot\left(\sigma_j(\bM)\right)^{p/2}}.$$
At last apply the Cauchy-Schwartz inequality to the vectors of singular values
$$\sum_{j=1}^{m-s}{1\cdot \left(\sigma_j(\bM)\right)^{p/2}}\leq \sqrt{m-s}\,
\sqrt{\sum_{j=1}^{m-s}{\left(\sigma_j(\bM)\right)^p}}\leq \sqrt{m-s}\,\spn{\bM}^{p/2}.$$
Merging the last two sequences of inequalities gives
$$\spnt{\bB\,\bM\,\bB}^{p/2}\leq \sqrt{m-s}\,\spn{\bM}^{p/2}.$$
Thus $\spnt{\bB\,\bM\,\bB}\leq \sqrt[p]{m-s}\,\spn{\bM}$, which can now be substituted into~(\ref{e_24}).
\end{compactenum}
\paragraph{3. Frobenius norm}
This is the special case $p=2$ with $\stn{\bA}=\|\bA\|_F$ and $\son{\bA}=\|\bA\|_*$.
\end{proof}

\subsection{Approximation error, and  error matrix}\label{s_errmatrix}
We generalize  \cite[Theorems 2 and 3]{DKM06} to Schatten $p$-norms.

\begin{theorem}\label{t_lc}
Let $\bA\in\rmn$ with $\rank(\bA)\geq k$; $\bC\in\real^{m\times c}$ with $\rank(\bC)=c\geq k$;
and let $p\geq 1$ be an even integer. Then
\begin{compactenum}
\item Two-norm $(p=\infty)$
$$\|(\bI-\bC\bC^{\dagger})\,\bA\|_{2}^2 \ \leq\  \|\bA\bA^T -\bC\bC^T\|_{2}.$$
\item Schatten $p$-norm $(p$ even$)$
$$\spn{(\bI-\bC\bC^{\dagger})\,\bA}^2 \ \leq\
\ \min\bigg\{ \spnt{\bA\bA^T-\bC\bC^T}, \quad \sqrt[p]{m-c}\,\spn{\bA\bA^T-\bC\bC^T}\bigg\}.$$
\item Frobenius norm $(p=2)$
$$\|(\bI-\bC\bC^{\dagger})\,\bA\|_{F}^2 \ \leq\
\ \min\bigg\{ \|\bA\bA^T-\bC\bC^T\|_*, \quad \sqrt{m-c}\,\|\bA\bA^T-\bC\bC^T\|_F\bigg\}.$$
\end{compactenum}
\end{theorem}

\begin{proof}
This follows from Theorem~\ref{t_perturbmm} with $\proj=\bC\bC^{\dagger}$, $\rank(\proj)=c$, $\wA=\bC$, and
$$(\bI-\proj)\,\wA=(\bI-\bC\bC^{\dagger})\,\bC=\bC=\bC\bC^{\dagger}\bC=\bzero.$$
\end{proof}

Recall
\textit{Mirsky's Theorem} \cite[Corollary 7.4.9.3]{HoJ12}, an extension of the Hoffman-Wielandt theorem
to any unitarily invariant norm and, in particular, Schatten $p$-norms: for $\bA,\bH\in\rmn$, the singular values $\sigma_j(\bA\bA^T)$ and $\sigma_j(\bH\bH^T)$, $1\leq j\leq m$
are also eigenvalues and satisfy
\begin{equation}\label{e_M}
\sum_{j=1}^m{|\sigma_j(\bA\bA^T)-\sigma_j(\bH\bH^T)|^p}\leq \spn{\bA\bA^T-\bH\bH^T}^p.
\end{equation}

\begin{theorem}\label{t_lck}
Let $\bA\in\rmn$ with $\rank(\bA)\geq k$; let $\bC\in\real^{m\times c}$ with $\rank(\bC)=c\geq k$
and best rank-$k$ approximation $\bC_k$; and let $p\geq 1$ be an even integer. Then
\begin{compactenum}
\item Two-norm $(p=\infty)$
$$\|(\bI-\bC_k\bC_k^{\dagger})\,\bA\|_{2}^2 \ \leq\  \|\bA-\bA_k\|_2^2+ 2\,\|\bA\bA^T -\bC\bC^T\|_{2}.$$
\item Schatten $p$-norm $(p$ even$)$
\begin{eqnarray*}
\spn{(\bI-\bC_k\bC_k^{\dagger})\,\bA}^2  &\leq& \spn{\bA-\bA_k}^2+\\
&& 2\,\min\big\{ \spnt{\bA\bA^T-\bC\bC^T}, \, \sqrt[p]{m-c}\,\spn{\bA\bA^T-\bC\bC^T}\big\}.
\end{eqnarray*}
\item Frobenius norm $(p=2)$
\begin{eqnarray*}
\|(\bI-\bC_k\bC_k^{\dagger})\,\bA\|_{F}^2 \ &\leq&\ \|\bA-\bA_k\|_F^2+\\
&&2\, \min\big\{ \|\bA\bA^T-\bC\bC^T\|_*, \, \sqrt{m-c}\,\|\bA\bA^T-\bC\bC^T\|_F\big\}.
\end{eqnarray*}
\end{compactenum}
\end{theorem}

\begin{proof}
We first introduce some notation before proving the bounds.

\paragraph{0. Set up}
Partition $\bA=\bA_k+\bA_{\perp}$ and $\bC=\bC_k+\bC_{\perp}$ to distinguish the
respective best rank-$k$ approximations $\bA_k$ and $\bC_k$.
From $\bA_k\bA_{\perp}^T=\bzero$ and $\bC_k\bC_{\perp}^T=\bzero$ follows
\begin{equation}
\bA\bA^T=\bA_k\bA_k^T+\bA_{\perp}\bA_{\perp}^T, \qquad
\bC\bC^T=\bC_k\bC_k^T+\bC_{\perp}\bC_{\perp}^T.
\end{equation}
Since the relevant matrices are symmetric positive semi-definite,
eigenvalues are equal to singular values. The dominant ones are
 $$\sigma_j(\bA_k\bA_k^T)=\sigma_j(\bA\bA^T)=\sigma_j(\bA)^2,\quad
 \sigma_j(\bC_k\bC_k^T)=\sigma_j(\bC\bC^T)=\sigma_j(\bC)^2, \quad 1\leq j \leq k,$$
 and the subdominant ones are, with $j \geq 1$,
 $$\sigma_j(\bA_{\perp}\bA_{\perp}^T)=\sigma_{k+j}(\bA\bA^T) =\sigma_{k+j}(\bA)^2,\quad
 \sigma_j(\bC_{\perp}\bC_{\perp}^T)=\sigma_{k+j}(\bC\bC^T)=\sigma_{k+j}(\bC)^2.$$
To apply Theorem~\ref{t_perturbmm}, set $\wA=\bC$,
$\proj=\bC_k\bC_k^{\dagger}$, $\rank(\proj_k)=k$, so that
$$(\bI-\proj)\,\wA=(\bI-\bC_k\bC_k^{\dagger})\,(\bC_k+\bC_{\perp})=\bC_{\perp}.$$
Thus
\begin{equation}\label{e_cperp}
\spn{(\bI-\proj)\,\wA}^2=\spn{\bC_{\perp}}^2.
\end{equation}

\paragraph{Two-norm}
Substituting (\ref{e_cperp}) into the two norm bound in Theorem~\ref{t_perturbmm} gives
\begin{equation}\label{e_two}
\|(\bI-\bC_k\bC_k^{\dagger})\,\bA\|_{2}^2 \leq  \|\bC_{\perp}\|_2^2 + \|\bA\bA^T -\bC\bC^T\|_{2}.
\end{equation}
The above and Weyl's theorem imply
\begin{eqnarray*}
\|\bC_{\perp}\|_2^2&=&\|\bC_{\perp}\bC_{\perp}^T\|_2=\lambda_{k+1}(\bC\bC^T)\\
&\leq& \left|\lambda_{k+1}(\bC\bC^T)-\lambda_{k+1}(\bA\bA^T)\right|+\lambda_{1}(\bA_{\perp}\bA_{\perp}^T)\\
&\leq& \|\bA\bA^T-\bC\bC^T\|_2 +\|\bA_{\perp}\|_2^2.
\end{eqnarray*}
Substituting this into (\ref{e_two}) gives
$$\|(\bI-\bC_k\bC_k^{\dagger})\,\bA\|_{2}^2  \leq \|\bA-\bA_k\|_2^2 +  2\, \|\bA\bA^T -\bC\bC^T\|_{2}.$$

\paragraph{Schatten $p$-norm $(p$ even$)$}
Substituting (\ref{e_cperp}) into the Schatten-$p$ norm bound in Theorem~\ref{t_perturbmm} gives
\begin{eqnarray}\label{e_p}
\spn{(\bI-\bC_k\bC_k^{\dagger})\,\bA}^2 &\leq &\spn{\bC_{\perp}}^2+ \\
&&\min\big\{ \spnt{\bA\bA^T-\bC\bC^T}, \, \sqrt[p]{m-c}\,\spn{\bA\bA^T-\bC\bC^T}\big\}.\nonumber
\end{eqnarray}
From (\ref{e_24}) follows $\spn{\bC_{\perp}}^2=\spnt{\bC_{\perp}\bC_{\perp}^T}$.
For a column vector $\bx$, let
$$\|\bx\|_p=\sqrt[p]{\sum_j{|x_j|^{1/p}}}$$
be the ordinary vector $p$-norm, and put the singular values of $\bC_{\perp}\bC_{\perp}^T$
into the vector
$$\bc_{\perp}\equiv\begin{pmatrix} \sigma_{1}(\bC_{\perp}\bC_{\perp}^T) & \cdots &
\sigma_{m-k}(\bC_{\perp}\bC_{\perp}^T)\end{pmatrix}^T.$$
Move from matrix norm to vector norm,
$$\spnt{\bC_{\perp}\bC_{\perp}^T}^{p/2}=\sum_{j=1}^{m-k}{\sigma_j(\bC_{\perp}\bC_{\perp}^T)^{p/2}}
=\sum_{j=1}^{m-k}{c_j^{p/2}}=\spnt{\bc_{\perp}}^{p/2}.$$
Put the singular values of $\bA_{\perp}\bA_{\perp}^T$ into the vector
$$\ba_{\perp}\equiv\begin{pmatrix} \sigma_{1}(\bA_{\perp}\bA_{\perp}^T) & \cdots &
\sigma_{m-k}(\bA_{\perp}\bA_{\perp}^T)\end{pmatrix}^T,$$
and apply the triangle inequality in the vector norm
$$\spnt{\bC_{\perp}\bC_{\perp}^T}=\spnt{\bc_{\perp}}
\leq \spnt{\bc_{\perp}-\ba_{\perp}}+\spnt{\ba_{\perp}}.$$
Substituting the following expression
$$\spnt{\ba_{\perp}}^{p/2}=\sum_{j=1}^{m-k}{\sigma_j(\bA_{\perp}\bA_{\perp}^T)^{p/2}}
=\sum_{j=1}^{m-k}{\sigma_j(\bA_{\perp})^p}=\spn{\bA_{\perp}}^p$$
into the previous bound and applying (\ref{e_24}) again gives
\begin{equation}\label{e_sp}
\spn{\bC_{\perp}}^2=\spnt{\bC_{\perp}\bC_{\perp}^T}\leq \spnt{\bc_{\perp}-\ba_{\perp}}+\spn{\bA_{\perp}}^2.
\end{equation}
\begin{compactenum}
\item First term in the minimum in (\ref{e_p}): Apply Mirsky's Theorem (\ref{e_M}) to the first summand in (\ref{e_sp})
\begin{eqnarray*}
\spnt{\bc_{\perp}-\ba_{\perp}}^{p/2}&=&
\sum_{j=1}^{m-k}{\left|\sigma_{k+j}(\bC\bC^T) -\sigma_{k+j}(\bA\bA^T)\right|^{p/2}}\\
&\leq&\sum_{j=1}^{m}{\left|\sigma_j(\bC\bC^T) -\sigma_j(\bA\bA^T)\right|^{p/2}}
\leq \spnt{\bC\bC^T-\bA\bA^T}^{p/2}.
\end{eqnarray*}
Thus,
$$\spnt{\bc_{\perp}-\ba_{\perp}}\leq \spnt{\bC\bC^T-\bA\bA^T}.$$
Substitute this into (\ref{e_sp}), so that
\begin{equation*}
\spn{\bC_{\perp}}^2\leq \spn{\bA_{\perp}}^2 + \spnt{\bC\bC^T-\bA\bA^T},
\end{equation*}
and the result in turn into (\ref{e_p}) to obtain the first term in the minimum,
\begin{eqnarray*}
\spn{(\bI-\bC_k\bC_k^{\dagger})\,\bA}^2 \leq \spn{\bA_{\perp}}^2 + 2\, \spnt{\bC\bC^T-\bA\bA^T}.
\end{eqnarray*}
\item Second term in the minimum in (\ref{e_p}): Consider the first summand in (\ref{e_sp}), but apply
the Cauchy Schwartz inequality before Mirsky's Theorem~(\ref{e_M}),
\begin{eqnarray*}
\spnt{\bc_{\perp}-\ba_{\perp}}^{p/2}&=&
\sum_{j=1}^{m-k}{\left|\sigma_{k+j}(\bC\bC^T) -\sigma_{k+j}(\bA\bA^T)\right|^{p/2}}\\
&\leq& \sqrt{m-k}\>\sqrt{\sum_{j=1}^{m-k}{\left|\sigma_{k+j}(\bC\bC^T) -\sigma_{k+j}(\bA\bA^T)\right|^{p}}}\\
&\leq& \sqrt{m-k}\>\sqrt{\sum_{j=1}^{m}{\left|\sigma_{k+j}(\bC\bC^T) -\sigma_{k+j}(\bA\bA^T)\right|^{p}}}\\
&\leq &\sqrt{m-k}\>\spn{\bC\bC^T-\bA\bA^T}^{p/2}.
\end{eqnarray*}
Thus,
$$\spnt{\bc_{\perp}-\ba_{\perp}}\leq \sqrt[p]{m-k}\>\spn{\bC\bC^T-\bA\bA^T}.$$
Substitute this into (\ref{e_sp}), so that
\begin{equation*}
\spn{\bC_{\perp}}^2\leq \spn{\bA_{\perp}}^2 + \sqrt{m-k}\,\spn{\bC\bC^T-\bA\bA^T},
\end{equation*}
and the result in turn into (\ref{e_p}) to obtain the second term in the minimum,
\begin{eqnarray*}
\spn{(\bI-\bC_k\bC_k^{\dagger})\,\bA}^2 \leq \spn{\bA_{\perp}}^2 + 2\, \sqrt[p]{m-k}\>\spn{\bC\bC^T-\bA\bA^T}.
\end{eqnarray*}
\end{compactenum}
\paragraph{3. Frobenius norm}
This is the special case $p=2$ with $\stn{\bA}=\|\bA\|_F$ and $\son{\bA}=\|\bA\|_*$.
\end{proof}

%% file: sec3.tex
\section{Approximation errors and angles between subspaces}\label{s_bounds}
We consider approximations where the rank of the orthogonal projector is at least as large as 
the dimension of the 
dominant subspace, and relate the low-rank approximation error to the 
subspace angle between projector and target space.  
After reviewing assumptions and notation (Section~\ref{s_ass}), we bound the low-rank approximation error
in terms of the subspace angle from below (Section~\ref{s_lower}) and from above (Section~\ref{s_upper}).

\subsection{Assumptions}\label{s_ass}
Given $\bA \in \rmn$ with  a gap after the $k$th singular value,
\begin{equation*}
\|\bA\|_2=\sigma_1(\bA)\geq \cdots\geq \sigma_k(\bA)> 
\sigma_{k+1}(\bA)\geq \cdots \geq \sigma_r(\bA)\geq 0,\qquad
r\equiv\min\{m,n\}.
\end{equation*}
Partition the full SVD $\bA=\bU\bSigma\bV^T$ in Section~\ref{s_sv} to distinguish between dominant
and subdominant parts,
$$\bU=\begin{pmatrix}\bU_k & \bU_{\perp}\end{pmatrix}, \qquad 
\bV=\begin{pmatrix} \bV_k & \bV_{\perp}\end{pmatrix}, \qquad
\bSigma=\diag\begin{pmatrix}\bSigma_k & \bSigma_{\perp}\end{pmatrix},$$
where the dominant parts are 
$$\bSigma_k\equiv\diag\begin{pmatrix}\sigma_1(\bA) & \cdots & \sigma_k(\bA)\end{pmatrix} \in\real^{k\times k},
\qquad \bU_k\in\real^{m\times k}, \qquad \bV_k\in\real^{n\times k},$$
and the subdominant ones 
\begin{equation*}
\bSigma_{\perp}\in\real^{(m-k)\times (n-k)}, 
\qquad \bU_{\perp}\in\real^{m\times (m-k)}, \qquad \bV_{\perp}\in\real^{n\times (n-k)}.
\end{equation*}
Thus $\bA$ is a "direct sum" 
$$\bA=\bA_k+\bA_{\perp}\qquad \text{where} \qquad
\bA_k\equiv \bU_k\bSigma_k\bV_k^T, \quad \bA_{\perp}\equiv \bU_{\perp}\bSigma_{\perp}\bV_{\perp}$$
and
\begin{equation}\label{e_orth}
\bA_{\perp}\bA_k^{\dagger} =\bzero=\bA_{\perp}\bA_k^T.
\end{equation}
The goal is to approximate the $k$-dimensional dominant left singular vector space,
\begin{equation}\label{e_proj}
\proju \equiv \bU_k\bU_k^T=\bA_k\bA_k^{\dagger}.
\end{equation}
To this end, let $\proj\in\rmm$ be an orthogonal projector as in (\ref{e_dproj}),
whose rank is at least as large as the dimension of the targeted subspace,
$$\rank(\proj)\geq \rank(\proju).$$
\subsection{Subspace angle as a lower bound for the approximation error}\label{s_lower}
We bound the low-rank approximation error from below by the subspace angle 
and the $k$th singular value of $\bA$, in the two-norm and the Frobenius norm.

\begin{theorem}\label{t_lau}
With the assumptions in Section~\ref{s_ass},
\begin{equation*}
\|(\bI-\proj)\bA\|_{2,F} \geq \sigma_k(\bA)\,\|\sin{\bTheta}(\proj,\proju)\|_{2,F}. 
\end{equation*}
\end{theorem}

\begin{proof}
From Lemma~\ref{l_projbasis}, (\ref{e_proj}), and (\ref{e_orth})  follows
\begin{eqnarray*}
\|\sin{\bTheta}(\proj,\proju)\|_{2,F}&=&\|(\bI-\proj)\,\proju\|_{2,F}=\|(\bI-\proj)\,\bA_k\bA_k^\dagger\|_{2,F}\\
&=& \|(\bI-\proj)\,(\bA_k+\bA_{\perp})\,\bA_k^\dagger\|_{2,F}= \|(\bI-\proj)\,\bA\bA_k^\dagger\|_{2,F}\\
& \leq &\|\bA_k^{\dagger}\|_2\, \|(\bI-\proj)\,\bA\|_{2,F}=\|(\bI-\proj)\bA\|_{2,F}/\sigma_k(\bA).
\end{eqnarray*}
\end{proof}
\subsection{Subspace angle as upper bound for the approximation error}\label{s_upper}
We present upper bounds for the low-rank approximation error in terms of the subspace
angle, the two norm (Theorem~\ref{t_lal1}) and Frobenius norm (Theorem~\ref{t_lal2}).

The bounds are guided by the following observation.
In the ideal case, where $\proj$ completely captures the target space, we have
$\range(\proj)=\range(\proju)=\range(\bA_k)$,
and 
\begin{equation*}
\|\sin{\bTheta}(\proj,\proju)\|_{2,F}=0, \qquad
\|(\bI-\proj)\bA\|_{2,F}=\|\bA_{\perp}\|_{2,F}=\|\bSigma_{\perp}\|_{2,F},
\end{equation*}
thus suggesting an additive error in the general, non-ideal case. 

\begin{theorem}[Two-norm]\label{t_lal1}
With the assumptions in Section~\ref{s_ass}, 
\begin{eqnarray*}
\|(\bI-\proj)\bA\|_2\leq \|\bA\|_2\,\|\sin{\bTheta(\proj,\proju)}\|_2
+\|\bA-\bA_k\|_2\,\|\cos{\bTheta(\bI-\proj,\bI-\proju)}\|_2.
\end{eqnarray*}
If also $k<\rank(\proj)+k<m$, then
\begin{eqnarray*}
\|(\bI-\proj)\bA\|_2\leq \|\bA\|_2\,\|\sin{\bTheta(\proj,\proju)}\|_2 +\|\bA-\bA_k\|_2.
\end{eqnarray*}
\end{theorem}

\begin{proof}
From $\bA=\bA_k+\bA_{\perp}$ and the triangle inequality follows
\begin{eqnarray*}
\|(\bI-\proj)\bA\|_2&\leq&\|(\bI-\proj)\bA_k\|_2+ \|(\bI-\proj)\bA_{\perp}\|_2.
\end{eqnarray*}

\paragraph{First summand}
Since $\rank(\proj)\geq \rank(\proju)$, Lemma~\ref{l_projbasis} implies
\begin{eqnarray*}
\|(\bI-\proj)\bA_k\|_2&\leq& \|(\bI-\proj)\,\bU_k\|_2\|\bSigma_k\|_2=\|\bA\|_2\,\|(\bI-\proj)\proju\|_2\\
&=&\|\bA\|_2\, \|\sin{\bTheta}(\proj,\proju)\|_2
\end{eqnarray*}
Substitute this into the previous bound to obtain
\begin{eqnarray}\label{e_5}
\|(\bI-\proj)\bA\|_2&\leq&\|\bA\|_2\,\|\sin{\bTheta}(\proj,\proju)\|_2+ \|(\bI-\proj)\bA_{\perp}\|_2.
\end{eqnarray}

\paragraph{Second summand} 
Submultiplicativity implies
\begin{eqnarray*}
\|(\bI-\proj)\bA_{\perp}\|_2\leq
\|(\bI-\proj)\,\bU_{\perp}\|_2\,\|\bSigma_{\perp}\|_2=
\|\bA-\bA_k\|_2\>\|(\bI-\proj)\,\bU_{\perp}\|_2.
\end{eqnarray*}
Regarding the last factor, the full SVD of $\bA$ in Section~\ref{s_ass} implies 
$$\range(\bU_{\perp})=\range(\bU_{\perp}\bU_{\perp}^T)=\range(\bU_k\bU_k^T)^{\perp}=\range(\proju)^{\perp}
=\range(\bI-\proju)$$
so that 
\begin{equation*}
\|(\bI-\proj)\,\bU_{\perp}\|_2=\|(\bI-\proj)\,(\bI-\proju)\|_2=\|\cos{\bTheta(\bI-\proj,\bI-\proju)}\|_2.
\end{equation*}
Thus, 
$$\|(\bI-\proj)\bA_{\perp}\|_2\leq\|\bA-\bA_k\|_2\,\|\cos{\bTheta(\bI-\proj,\bI-\proju)}\|_2.$$
Substitute this into (\ref{e_5}) to obtain the first bound.

\paragraph{Special case $\rank(\proj)+k<m$}
From Corollary~\ref{c_csnequal} follows with $\ell\equiv \rank(\proj)$
\begin{eqnarray*}
\|\cos{\bTheta}(\bI-\proj,\bI-\proju)\|_2=
\Big\|\begin{pmatrix}\bI_{m-(k+\ell)} & \\ & \cos{\bTheta}(\proj,\proju)\end{pmatrix}\Big\|_2=1.
\end{eqnarray*}
\end{proof}

\begin{theorem}[Frobenius norm]\label{t_lal2}
With the assumptions in Section~\ref{s_ass} 
\begin{eqnarray*}
\|(\bI-\proj)\bA\|_2\leq \|\bA\|_2\,\|\sin{\bTheta(\proj,\proju)}\|_F+
\min\left\{\|\bA-\bA_k\|_2 \, \|\bGamma\|_F,\>\|\bA-\bA_k\|_F\,\|\bGamma\|_2\right\},
\end{eqnarray*}
where $\bGamma\equiv \cos{\bTheta(\bI-\proj,\bI-\proju)}$.

If also $k<\rank(\proj)+k<m$, then 
\begin{eqnarray*}
\|(\bI-\proj)\bA\|_F&\leq& \|\bA\|_2\,\|\sin{\bTheta(\proj,\proju)}\|_F +\|\bA-\bA_k\|_F.
\end{eqnarray*}
\end{theorem}

\begin{proof}
With strong submultiplicativity, the analogue of (\ref{e_5}) is
\begin{eqnarray}\label{e_5b}
\|(\bI-\proj)\bA\|_F&\leq&\|\bA\|_2\,\|\sin{\bTheta}(\proj,\proju)\|_F+ \|(\bI-\proj)\bA_{\perp}\|_F.
\end{eqnarray}
There are  two options to bound
$\|(\bI-\proj)\bA_{\perp}\|_F=\|(\bI-\proj)\,\bU_{\perp}\bSigma_{\perp}\|_F$, depending on which factor gets
the two norm.  Either
\begin{eqnarray*}
\|(\bI-\proj)\,\bU_{\perp}\bSigma_{\perp}\|_F\leq \|(\bI-\proj)\,\bU_{\perp}\|_F\,\|\bSigma_{\perp}\|_2
=\|\bA-\bA_k\|_2\,\|(\bI-\proj)\,\bU_{\perp}\|_F
\end{eqnarray*}
or
\begin{eqnarray*}
\|(\bI-\proj)\,\bU_{\perp}\bSigma_{\perp}\|_F&\leq &\|(\bI-\proj)\,\bU_{\perp}\|_2\,\|\bSigma_{\perp}\|_F
=\|\bA-\bA_k\|_F\,\|(\bI-\proj)\,\bU_{\perp}\|_2.
\end{eqnarray*}
As in the proof of Theorem~\ref{t_lal1} one shows
\begin{equation*}
\|(\bI-\proj)\,\bU_{\perp}\|_{2,F}=\|\cos{\bTheta(\bI-\proj,\bI-\proju)}\|_{2,F},
\end{equation*}
as well as the expression for the special case $\rank(\proj)+k<m$.
\end{proof}

%% file: secapp.tex
\section{CS Decompositions}\label{s_app}
We review expressions for the CS decompositions from
\cite[Theorem 8.1]{PaigeWei94} and \cite[Section 2]{ZKny13}.

We consider a subspace $\range(\bZ)$ of dimension $k$, and a subspace of $\range(\hZ)$ of dimension $\ell\geq k$,
whose dimensions sum up to less than the dimension of the host space.

Let $\begin{pmatrix} \bZ & \bZ_{\perp}\end{pmatrix},
\begin{pmatrix} \hZ & \hZ_{\perp}\end{pmatrix} \in\rmm$  be orthogonal matrices where
$\bZ\in\real^{m\times k}$ and $\hZ\in\real^{m\times \ell}$ with $\ell\geq k$.
Let
$$\begin{pmatrix} \bZ & \bZ_{\perp}\end{pmatrix}^T\,
\begin{pmatrix} \hZ & \hZ_{\perp}\end{pmatrix} =
\begin{pmatrix} \bZ^T\hZ & \bZ^T\hZ_{\perp} \\ \bZ_{\perp}^T\hZ & \bZ_{\perp}^T\hZ_{\perp}\end{pmatrix}=
\begin{pmatrix}\bQ_{11} & \\ &\bQ_{12}\end{pmatrix}\, \bD\,
\begin{pmatrix}\bQ_{21} & \\ & \bQ_{22}\end{pmatrix}$$
be a CS decomposition where $\bQ_{11}\in\real^{k\times k}$, $\bQ_{12}\in\real^{(m-k)\times (m-k)}$,
$\bQ_{21} \in\real^{\ell\times \ell}$ and $\bQ_{22}\in\real^{(m-\ell)\times (m-\ell)}$ are all orthogonal matrices.

\begin{theorem}\label{t_csnequal}
If $k<\ell<m-k$ then
\begin{equation*}
\bD \ = \ \begin{blockarray}{ccccccc}
r&s&\ell-(r+s)&m-(k+\ell)+r&s& k-(r+s)\\
\begin{block}{[ccc|ccc]c}
\bI_r &  & &\bzero& & &r \\
&\bC & & &\bS&&s\\
 & & \bzero & & &\bI_{k-(r+s)} &k-(s+r)\\
\BAhline
\bzero& & &-\bI_{m-(k+\ell)+r}&&&m-(k+\ell)+r\\
  & \bS &  & & -\bC &&s  \\
 && \bI_{\ell-(r+s)} &  & & \bzero& \ell-(r+s)\\
\end{block}
\end{blockarray}.
\end{equation*}
Here $\bC^2+\bS^2=\bI_s$ with
$$\bC=\diag\begin{pmatrix}\cos{\theta_1} & \cdots & \cos{\theta_s} \end{pmatrix}, \quad
\bS=\diag\begin{pmatrix}\sin{\theta_1} & \cdots & \sin{\theta_s}\end{pmatrix},$$
and
\begin{eqnarray*}
&&r= \dim\left(\range(\bZ)\cap\range(\hZ)\right), \quad
m-(k+\ell)+r=\dim\left(\range(\bZ_{\perp})\cap\range(\hZ_{\perp})\right)\\
&&\ell-(r+s) = \dim\left(\range(\bZ_{\perp})\cap\range(\hZ)\right), \
k-(r+s)=\dim\left(\range(\bZ)\cap\range(\hZ_{\perp})\right).
\end{eqnarray*}
\end{theorem}

\begin{corollary}\label{c_csnequal}
From Theorem~\ref{t_csnequal} follows
\begin{eqnarray*}\label{e_csequal1}
\|\sin{\bTheta}(\bZ,\hZ)\|_{2,F}&=&\|\bZ^T\hZ_{\perp}\|_{2,F}
=\Big\|\begin{pmatrix}\bS & \\ & \bI_{k-(r+s)}\end{pmatrix}\Big\|_{2,F}\\
\|\cos{\bTheta}(\bZ,\hZ)\|_{2,F}&=&\|\bZ^T\hZ\|_{2,F}=\Big\|\begin{pmatrix}\bI_r & \\ & \bC\end{pmatrix}\Big\|_{2,F}\\
\|\cos{\bTheta}(\bZ_{\perp},\hZ_{\perp})\|_{2,F}&=&\|\bZ_{\perp}^T\hZ_{\perp}\|_{2,F}=
\Big\|\begin{pmatrix}\bI_{m-(k+\ell)} & \\ & \cos{\bTheta}(\bZ,\hZ)\end{pmatrix}\Big\|_{2,F}.
\end{eqnarray*}
\end{corollary}